\documentclass[11pt,a4paper]{article}
\usepackage{graphicx}
\usepackage{mathrsfs}
\usepackage{color}
\usepackage[latin1]{inputenc}
\usepackage[francais, english]{babel}

\usepackage{libertine}
\usepackage[T1]{fontenc}
\usepackage{lipsum}
\usepackage{titlesec}
\usepackage[colorlinks=true,linkcolor=magenta,citecolor=blue]{hyperref}

\usepackage[cmex10]{amsmath}
\usepackage{stmaryrd}
\usepackage{color}
\usepackage{cmll}
\usepackage{amsmath,amssymb}
\usepackage[sloped]{fourier}
\usepackage{amsthm}
\usepackage{esint}
\usepackage{psfrag}

\author{S{\'e}bastien Alvarez\thanks{The author was supported by a post-doctoral grant financed by CAPES at IMPA, Rio de Janeiro.}}
\date{}
\title{A Fatou theorem for $F$-harmonic functions}

\begin{document}

\newtheorem{maintheorem}{Theorem}
\newtheorem{maincoro}[maintheorem]{Corollary}
\renewcommand{\themaintheorem}{\Alph{maintheorem}}
\newcounter{theorem}[section]
\newtheorem{exemple}{\bf Exemple \rm}
\newtheorem{exercice}{\bf Exercice \rm}
\newtheorem{conj}[theorem]{\bf Conjecture}
\newtheorem{defi}[theorem]{\bf Definition}
\newtheorem{lemma}[theorem]{\bf Lemma}
\newtheorem{proposition}[theorem]{\bf Proposition}
\newtheorem{coro}[theorem]{\bf Corollary}
\newtheorem{theorem}[theorem]{\bf Theorem}
\newtheorem{rem}[theorem]{\bf Remark}
\newtheorem{ques}[theorem]{\bf Question}
\newtheorem{propr}[theorem]{\bf Property}
\newtheorem{question}{\bf Question}
\def\bp{\noindent{\it Proof. }}
\def\ep{\noindent{\hfill $\fbox{\,}$}\medskip\newline}
\renewcommand{\theequation}{\arabic{section}.\arabic{equation}}
\renewcommand{\thetheorem}{\arabic{section}.\arabic{theorem}}
\newcommand{\eps}{\varepsilon}
\newcommand{\disp}[1]{\displaystyle{\mathstrut#1}}
\newcommand{\fra}[2]{\displaystyle\frac{\mathstrut#1}{\mathstrut#2}}
\newcommand{\dif}{{\rm Diff}}
\newcommand{\homeo}{{\rm Homeo}}
\newcommand{\Per}{{\rm Per}}
\newcommand{\Fix}{{\rm Fix}}
\newcommand{\A}{\mathcal A}
\newcommand{\Z}{\mathbb Z}
\newcommand{\Q}{\mathbb Q}
\newcommand{\R}{\mathbb R}
\newcommand{\C}{\mathbb C}
\newcommand{\N}{\mathbb N}
\newcommand{\T}{\mathbb T}
\newcommand{\U}{\mathbb U}
\newcommand{\D}{\mathbb D}
\newcommand{\PP}{\mathbb P}
\newcommand{\Sp}{\mathbb S}
\newcommand{\K}{\mathbb K}
\newcommand{\car}{\mathbf 1}
\newcommand{\g}{\mathfrak g}
\newcommand{\gs}{\mathfrak s}
\newcommand{\h}{\mathfrak h}
\newcommand{\rr}{\mathfrak r}
\newcommand{\fhi}{\varphi}
\newcommand{\ffhi}{\tilde{\varphi}}
\newcommand{\moins}{\setminus}
\newcommand{\ds}{\subset}
\newcommand{\W}{\mathcal W}
\newcommand{\WW}{\widetilde{W}}
\newcommand{\F}{\mathcal F}
\newcommand{\G}{\mathcal G}
\newcommand{\CC}{\mathcal C}
\newcommand{\RR}{\mathcal R}
\newcommand{\DD}{\mathcal D}
\newcommand{\M}{\mathcal M}
\newcommand{\B}{\mathcal B}
\newcommand{\V}{\mathcal V}
\newcommand{\cS}{\mathcal S}
\newcommand{\HH}{\mathcal H}
\newcommand{\Hyp}{\mathbb H}
\newcommand{\UU}{\mathcal U}
\newcommand{\OO}{\mathcal O}
\newcommand{\Pp}{\mathcal P}
\newcommand{\QQ}{\mathcal Q}
\newcommand{\E}{\mathcal E}
\newcommand{\GG}{\Gamma}
\newcommand{\LL}{\mathcal L}
\newcommand{\KK}{\mathcal K}
\newcommand{\TT}{\mathcal T}
\newcommand{\X}{\mathcal X}
\newcommand{\Y}{\mathcal Y}
\newcommand{\ZZ}{\mathcal Z}
\newcommand{\bE}{\overline{E}}
\newcommand{\bF}{\overline{F}}
\newcommand{\wF}{\widetilde{F}}
\newcommand{\hcF}{\widehat{\mathcal F}}
\newcommand{\bW}{\overline{W}}
\newcommand{\bcW}{\overline{\mathcal W}}
\newcommand{\tcW}{\widetilde{\mathcal W}}
\newcommand{\tL}{\widetilde{L}}
\newcommand{\diam}{{\rm diam}}
\newcommand{\diag}{{\rm diag}}
\newcommand{\Jac}{{\rm Jac}}
\newcommand{\Plong}{{\rm Plong}}
\newcommand{\Tr}{{\rm Tr}}
\newcommand{\Conv}{{\rm Conv}}
\newcommand{\Cl}{{\rm Cl}}
\newcommand{\Ext}{{\rm Ext}}
\newcommand{\Spec}{{\rm Sp}}
\newcommand{\Isom}{{\rm Isom}\,}
\newcommand{\Supp}{{\rm Supp}\,}
\newcommand{\Grass}{{\rm Grass}}
\newcommand{\Hold}{{\rm H\ddot{o}ld}}
\newcommand{\Ad}{{\rm Ad}}
\newcommand{\Area}{{\rm Area}}
\newcommand{\ad}{{\rm ad}}
\newcommand{\e}{{\rm e}}
\newcommand{\s}{{\rm s}}
\newcommand{\pol}{{\rm pole}}
\newcommand{\Aut}{{\rm Aut}}
\newcommand{\End}{{\rm End}}
\newcommand{\Leb}{{\rm Leb}}
\newcommand{\Liouv}{{\rm Liouv}}
\newcommand{\Lip}{{\rm Lip}}
\newcommand{\Int}{{\rm Int}}
\newcommand{\cc}{{\rm cc}}
\newcommand{\grad}{{\rm grad}}
\newcommand{\proj}{{\rm proj}}
\newcommand{\mass}{{\rm mass}}
\newcommand{\dive}{{\rm div}}
\newcommand{\dist}{{\rm dist}}
\newcommand{\im}{{\rm Im}}
\newcommand{\re}{{\rm Re}}
\newcommand{\codim}{{\rm codim}}
\newcommand{\Map}{\longmapsto}
\newcommand{\vide}{\emptyset}
\newcommand{\tr}{\pitchfork}
\newcommand{\ssl}{\mathfrak{sl}}

\newenvironment{demo}{\noindent{\textbf{Proof.}}}{\quad \hfill $\square$}
\newenvironment{pdemo}{\noindent{\textbf{Proof of the proposition.}}}{\quad \hfill $\square$}
\newenvironment{IDdemo}{\noindent{\textbf{Idea of proof.}}}{\quad \hfill $\square$}

\def\to{\mathop{\rightarrow}}
\def\act{\mathop{\curvearrowright}}
\def\To{\mathop{\longrightarrow}}
\def\Sup{\mathop{\rm Sup}}
\def\Max{\mathop{\rm Max}}
\def\Inf{\mathop{\rm Inf}}
\def\Min{\mathop{\rm Min}}
\def\lims{\mathop{\overline{\rm lim}}}
\def\limi{\mathop{\underline{\rm lim}}}
\def\egal{\mathop{=}}
\def\dans{\mathop{\subset}}
\def\surj{\mathop{\twoheadrightarrow}}

\def\h#1#2#3{ {\rm{hol}}^{#1}_{#2\rightarrow#3}}

\maketitle

\begin{abstract}
In this paper we study a class of functions that appear naturally in some equidistribution problems and that we call $F$-harmonic. These are functions of the universal cover of a closed and negatively curved manifold which possess an integral representation analogous to the Poisson representation of harmonic functions, where the role of the Poisson kernel is played by a H{\"o}lder continuous kernel. More precisely we prove a theorem {\`a} la Fatou about the nontangential convergence of quotients of such functions, from which we deduce some basic properties such as the uniqueness of the $F$-harmonic function on a compact manifold and of the integral representation of $F$-harmonic functions.
\end{abstract}

\section{Introduction}

\paragraph{Equidistribution problems.} This paper is devoted to the study of a class of functions that appear naturally in the resolution of some equidistribution problems. More precisely, consider a closed and negatively curved Riemannian manifold $M$ and a projective action of its fundamental group on the Riemann sphere $\rho:\pi_1(M)\to PSL_2(\C)$.

Denote by $N$ its Riemannian universal cover and fix a base point $o\in N$. We define a distance on $\pi_1(M)$ by $d(\gamma_1,\gamma_2)=\dist(\gamma_1 o,\gamma_2 o)$, and denote the ball centered at $Id$ and of radius $R$ by $B_R$. The equidistribution problem concerns the existence of the following limit
$$\lim_{R\to\infty}\frac{1}{|B_R|}\sum_{\gamma\in B_R}\delta_{\rho(\gamma) x},$$
where $x\in\C\PP^1$.

This discrete equidistribution problem possesses a continuous version in terms of foliations. By \emph{suspension} of $\rho$ it is possible to define a fiber bundle $\Pi:E\to M$ with fiber $\C\PP^1$ endowed with a foliation $\F$ transverse to the fibers whose holonomy representation is precisely given by $\rho$. Here we are interested in a multidimensional analogue of Birkhoff averages. Denote by $L_x$ the leaf of a point $x\in M$ and by $\proj_x:(N,o)\to(L_x,x)$ the Riemannian universal cover of the leaf. The equidistribution problem of large balls tangent to the leaves concerns the existence of the following limit
$$\lim_{R\to\infty}\proj_x\,_{\ast}\left(\frac{\Leb_{|B(o,R)}}{\Leb(B(o,R))}\right).$$

It is proven in \cite{Al3} (see also \cite{BG} for the case where the base is a hyperbolic surface) that when no measure on $\C\PP^1$ is invariant by the action of the holonomy group $\rho(\pi_1(M))$, each of these two equidistribution problems has a unique solution (independent of $x\in E$). These problems are related: disintegrating the solution of the second problem gives the solution of the first one.

The proof provides more: the solution of the second problem is a measure that has a local characterization similar to that of Garnett's \emph{harmonic measures} \cite{Gar}. Locally it is the product of a measure on $\C\PP^1$ by measures on the plaques that have densities with respect to Lebesgue, but the local densities are not necessarily harmonic (except in the constant curvature case). However they possess an integral representation very similar to the \emph{Poisson representation} of positive harmonic functions in negative curvature (see \cite{AS}). The exponential of the \emph{Busemann cocycle} plays the role of the \emph{Poisson kernel}: we called $0$-\emph{harmonic} this type of functions.

\paragraph{$F$-harmonic functions.} There are weighted versions of these equidistribution problems (see Paragraph \ref{examples} for the details) which led to introduce the notion of $F$-\emph{harmonic functions}. More precisely let $F:T^1M\to\R$ be a H{\"o}lder continuous function (also called a \emph{potential} in the thermodynamic formalism) and $\wF:T^1N\to\R$ be its lift to $N$. Denote by $P(F)$ its pressure (see Section \ref{gibbsstates} for the definition). Consider the \emph{Gibbs kernel}
$$k^F(y,z;\xi)=\exp\left[\int_{\xi}^{z}\widetilde{F}-\int_{\xi}^{y}\widetilde{F}\right]\exp\left[-P(F)\beta_{\xi}(y,z)\right],$$
where $\beta_{\xi}$ denotes the Busemann cocycle at $\xi$ and the difference of the integrals stands for a limit that we shall describe later. This kernel plays the role of the usual Poisson kernel, and we define a class of functions called $F$-\emph{harmonic} that consists of functions with the following integral representation
$$h(z)=\int_{N(\infty)} k^F(o,z;\xi)d\eta(\xi),$$
where $o$ is a base point and $\eta$ is a finite Radon measure on the sphere at infinity $N(\infty)$.

\paragraph{A theorem {\`a} la Fatou.} The aim of this article is to study these functions under the point of view of potential theory in order to show that this notion is not as artificial as it may seem at first sight. Our main result is to prove that they satisfy a theorem analogous to Fatou's theorem about nontangential convergence of harmonic functions (the version given by Anderson-Schoen \cite{AS} in the context of pinched negative curvature is pertinent here). More precisely it is a relativized version of Doob \cite{D} about the nontangential convergence of \emph{quotients} of harmonic functions that we will generalize to $F$-harmonic functions.

\begin{maintheorem}[{\`a} la Fatou]
\label{theoremedefatou}
Let $M$ be a closed Riemannian manifold with negative sectional curvature and $N$ be its Riemannian universal cover. Let $F:T^1M\to\R$ be a H{\"o}lder continuous potential. Finally let $h_1,h_2:N\to\R$ two $F$-harmonic functions defined by the integral formulae
$$h_i(z)=\int_{N(\infty)} k^F(o,z;\xi)d\eta_i(\xi),$$
where $\eta_i$, $i=1,2$ are finite Radon measures on $N(\infty)$. Write the Lebesgue decomposition of $\eta_1$ with respect to $\eta_2$ as
$$\eta_1=f\eta_2+\eta_s,$$
where $f$ is an $\eta_2$-integrable function, and $\eta_s$ is a finite Radon measure which is singular with respect to $\eta_2$. Then
\begin{enumerate}
\item for $\eta_2$-almost every $\xi\in N(\infty)$, we have
$$\frac{h_1(z)}{h_2(z)}\to f(\xi),$$
as $z$ converges nontangentially to $\xi$;
\item for $\eta_s$-almost every $\xi\in N(\infty)$, we have
$$\frac{h_1(z)}{h_2(z)}\to\infty$$
as $z$ converges nontangentially to $\xi$.
\end{enumerate}
\end{maintheorem}

The idea of looking at the theory of H{\"o}lder cocycles in negative curvature from a potential theoretic point of view is not new. Recall first that, using techniques \emph{{\`a} la Patterson-Sullivan}, Ledrappier \cite{L} was able to find that (up to a multiplicative constant) there is a unique $\pi(M)$-equivariant family $(\nu_z^F)_{z\in N}$ of finite measures on $N(\infty)$ that all lie in the same class and whose Radon-Nikodym cocycle coincides with the Gibbs kernel $k^F$
$$k^F(y,z;\xi)=\frac{d\nu_z^F}{d\nu_y^F}(\xi).$$

This construction has been generalized much beyond the case of compact manifolds: see \cite{PPS} for an optimal presentation of this theory. A priori in the noncompact case these densities are not unique anymore. In \cite{R1,PPS} Roblin (in the constant potential case), and Paulin-Pollicott-Schapira (in the general case) study the functions given by masses of invariant Patterson densities (as are called the measures $\nu_z^F$ in \cite{PPS}) and show an analogue of Fatou's theorem for these functions. Roblin \cite{R2} deduces a theory of Martin boundary for conformal densities.

Our theorem is not a consequence of the works of Roblin and Paulin-Pollicott-Schapira. Indeed their results concern functions which are invariant by some group of isometries (which is not necessarily cocompact). On the contrary, while we are interested in the cocompact case, our functions are not supposed to be invariant by a group of isometries. Still similarities exist with these works, and the tools used here are basically the same. We also derive some inequalities \emph{{\`a} la Harnack} from the now celebrated \emph{shadow lemma}. We also need to apply to shadows the usual abstract theory of measure differentiation which is usually stated using balls in $\R^n$. Finally we also make an intensive use of two of the cornerstones of modern geometric ergodic theory: distortion controls and comparison with constant curvature.

\paragraph{Basic properties.} Some basic properties of $F$-harmonic functions may be deduced from our main theorem. The first one concerns the uniqueness of the $F$-harmonic function on a compact manifold, up to a multiplicative constant. This is the analogue in our context of the fact that the only harmonic functions on a compact manifold are the constant ones. Notice that by definition $k^F$ is the Radon-Nikodym cocycle of the family of Ledrappier's measures. Hence the function $z\mapsto\mass(\nu_z^F)$ is a $F$-harmonic function. By equivariance it descends to the quotient and provides an example of $F$-harmonic function on $M$. The following theorem states that it is the only one up to a multiplicative constant. Note that we gave a shorter proof of this theorem in \cite{Al3} using the uniqueness of Gibbs states for the geodesic flow.

\begin{maintheorem}[Uniqueness of the $F$-harmonic function]
\label{uniqueFharmonic}
Let $M$ be a closed and negatively curved Riemannian manifold, $N$ be its Riemannian universal cover and $F:T^1M\to\R$ be a H{\"o}lder continuous potential. Let $(\nu_z^F)_{z\in N}$ be the family of Ledrappier's measures on $N(\infty)$ and $h_0^F$ be the function $z\mapsto\mass(\nu_z^F)$. Then the function $h_0^F$ is invariant by the action of $\pi_1(M)$ and descends to the unique $F$-harmonic function on $M$ (up to a multiplicative constant).
\end{maintheorem}

The following property was one of the principal motivations of the present work. Contrarily to the harmonic case our functions are not a priori invariant by a diffusion operator, and the uniqueness of the integral representation is not obvious at all.

\begin{maintheorem}[Uniqueness of the integral representation]
\label{uniquedecompositionFharmonic}
Let $M$ be a closed and negatively curved Riemannian manifold, $N$ be its Riemannian universal cover and $F:T^1M\to\R$ be a H{\"o}lder continuous potential.  Let $\eta_1,\eta_2$ be finite Radon measures on $N(\infty)$ such that the following equality holds for every $z\in N$
$$\int_{N(\infty)}k^F(o,z;\xi) d\eta_1(\xi)=\int_{N(\infty)}k^F(o,z;\xi) d\eta_2(\xi).$$
Then the two measures $\eta_1$ and $\eta_2$ are equal.
\end{maintheorem}

\paragraph{Outline of the paper.} In Section \ref{gibbsstatesgeoflow} we introduce $F$-harmonic functions and give the definitions that will be used throughout the text. In Section \ref{shadowscovborel} we introduce shadows and discuss the differentiation of Radon measures in $N(\infty)$. In Section \ref{redmax} we reduce the Theorem to the study of a maximal function {\`a} la Hardy-Littlewood. Section \ref{proofkey} is the technical part of this work: we prove the key proposition of the paper, i.e. the control of a quotient of $F$-harmonic functions by the maximal function. In Section \ref{propertiesexamples} 
the reader will find the proof of the basic properties of $F$-harmonic stated above as well as some examples of $F$-harmonic functions related to projective representations of fundamental groups of closed and negatively curved manifolds. Finally, following Tanaka \cite{T}, we prove in Appendix Borel's differentiation theorem in the sphere at infinity, in the more general context of Gromov hyperbolic spaces.

\section{Gibbs kernel and $F$-harmonic functions}
\label{gibbsstatesgeoflow}

\subsection{Negatively curved manifolds}
\label{sgeometry}
In all this article, $M$ will denote a closed (i.e. $C^{\infty}$ and boundaryless) Riemannian manifold with negative sectional curvature, and  $N$ its Riemannian universal cover. Upper and lower bounds for the sectional curvature of $M$ will be denoted respectively by $-a^2$ and $-b^2$. Finally when $\kappa<0$ we will denote by $N_{\kappa}$ the \emph{model space of curvature} $\kappa$, i.e. the unique complete, connected and simply connected space whose sectional curvature is everywhere $\kappa$.

\paragraph{Sphere at infinity.} Let $N(\infty)$ denote the \emph{sphere at infinity} of $N$, that is to say, the set of equivalence classes of geodesic rays for the relation ``stay at bounded distance''.

We consider $\pi_o:T^1_oN\to N(\infty)$ the natural projection which associates to $v$ the class of the geodesic ray it determines. All the maps $\pi_{o'}^{-1}\circ\pi_o:T^1_oN\to T^1_{o'}N$ are H{\"o}lder continuous: we have a natural H{\"o}lder structure on the sphere at infinity.

$N\cup N(\infty)$ is endowed with the \emph{cone topology}: neighbourhoods of infinity are given by the \emph{truncated cones} $T_o(v,\theta,R)=C_o(v,\theta)\moins B(o,R)$, where
$$C_o(v,\theta)=\{x\in N|\angle_o(v,v_{ox})\leq\theta\},$$
with $v\in T^1_oN$, $\theta\in\R$ and $v_{ox}$, the unit vector tangent to $o$ which points to $x$.

Moreover, one says that a sequence $(x_i)_{i\in\N}$ \emph{converges nontangentially} to a point at infinity $\xi$ if it converges in the cone topology while staying at bounded distance from a geodesic ray.

\paragraph{Action of isometries on the boundary.} The group of isometries $\Isom(N)$ acts on $T^1N$ by differential of isometries. Via the identifications $\pi_z:T_z^1N\to N(\infty)$, we deduce that there is a natural action of $\Isom(N)$ on the sphere $N(\infty)$.

\paragraph{Busemann cocycle and horospheres.}We define the \emph{Busemann cocycle} by the following
$$\beta_{\xi}(y,z)=\lim_{t\to\infty}\dist(c(t),z)-\dist(c(t),y),$$
where $\xi\in N(\infty)$, $y,z\in N$ and $c$ is any geodesic ray parametrized by length of arc and pointing to $\xi$. It satisfies the following cocycle relation
\begin{equation}
\label{Eq:cocyclebusemann}
\beta_{\xi}(x,y)+\beta_{\xi}(y,z)=\beta_{\xi}(x,z),
\end{equation}
for every $x,y,z\in N$ and $\xi\in N(\infty)$. It also satisfies the equivariance relation:
\begin{equation}
\label{Eq:equivariancebusemann}
\beta_{\gamma\xi}(\gamma y,\gamma z)=\beta_{\xi}(y,z),
\end{equation}
for every $y,z\in N$ and $\xi\in N(\infty)$.

The \emph{horospheres} are the level sets of this cocycle: two points $y,z$ are said to be on the same horosphere centered at $\xi$ if $\beta_{\xi}(y,z)=0$. It is possible (see the next paragraph) to see that horospheres are $C^{\infty}$ manifolds.

\subsection{The geodesic flow}
\label{gibbsstates}
\paragraph{Geodesic flow.} A vector $v\in T^1 M$ directs a unique geodesic. The \emph{geodesic flow} is defined by flowing $v$ along this geodesic at unit speed. We denote it by $g_t:T^1M\to T^1M$. It is well known that this flow is an \emph{Anosov flow} \cite{An}: there is a (H{\"o}lder) continuous splitting $TT^1M=E^s\oplus E^u\oplus\R X$, where $X$ is the generator of the flow and $E^s$, $E^u$ are $D g_t$-invariant subbundles which are respectively uniformly contracted and dilated by the flow

$$\begin{cases}\displaystyle{\frac{a}{b}e^{-bt}||v_s||\leq ||Dg_t(v_s)||\leq \frac{b}{a}e^{-at}||v_s||} & \text{if $v_s\in E^s$ and $t>0$} \\
                                                                                                                      \\
               \displaystyle{\frac{a}{b}e^{-bt}||v_u||\leq ||Dg_{-t}(v_u)||\leq \frac{b}{a}e^{-at}||v_u||} & \text{if $v_u\in E^u$ and $t>0$}
   .\end{cases}$$
The bundles $E^s$ and $E^u$ are respectively called \emph{stable} and \emph{unstable} bundles. They are uniquely integrable: we denote by $\W^s$, $\W^u$ the integral foliations called respectively \emph{stable} and \emph{unstable} foliations. Moreover, by the stable manifold theorem, the stable and unstable manifolds are as smooth as the geodesic flow that is $C^{\infty}$.

It is possible to lift the flow to $T^1N$ via the differential of the universal cover. We obtain a flow denoted by $G_t:T^1N\to T^1N$. The lifts of invariant foliations shall be denoted by $\tcW^{\star}$ and their leaves $W^{\star}(v)$, $\star=s,u,cs,cu$. Stable (resp. unstable) leaves are identified with stable (resp. unstable) horospheres, i.e. horospheres endowed with the normal inward (resp. outward) vector field. This proves that horospheres are smooth.

\paragraph{Pressure.} The \emph{pressure} of a H{\"o}lder continuous potential $F:T^1M\to\R$ is defined as:
$$P(F)=\lim_{\eps\to 0}\lims_{T\to\infty}\frac{1}{T}\log\,\Sup\left\{\sum_{p\in E}\exp\left[\int_0^TF\circ g_t(v)dt\right];\,E\,\,\,(T,\eps)\textrm{-separated}\right\},$$
where a set $E\dans T^1M$ is said to be $(T,\eps)$-separated if for every $v,w\in E$, $\Sup_{t\in[0,T]}\dist(g_t(v),g_t(w))\geq\eps$ (see \cite{BR}).

Recall that for every $\gamma\in\pi_1(M)$ there is a unique closed geodesic whose free homotopy class is conjugated to $\gamma$. Denote by $|\gamma|$ the length of this geodesic and by $\int_{\gamma} F$ the integral of $F$ along this geodesic. When $P(F)>0$ we have the following characterization of the pressure \cite{B,BR}

\begin{equation}
\label{pressureperiodic}
P(F)=\lim_{T\to\infty}\frac{1}{T}\log\sum_{\gamma\in\pi_1(M),|\gamma|\leq T}\exp\left(\int_{\gamma}F\right).
\end{equation}

\paragraph{Integrating the potential along trajectories.} The following lemma will be the main dynamical ingredient of our theorem {\`a} la Fatou. 
It relies upon the study of Poincar{\'e} series made in \cite{L}.

\begin{lemma}
\label{dynamicalingredient}
Let $F:T^1M\to\R$ be a H{\"o}lder continuous potential. Then there exists a constant $R_0>0$ such that for every $v\in T^1M$ and $T\geq R_0$ we have
$$\int_{0}^T[F\circ g_t(v)-P(F)] dt<0.$$
\end{lemma}

\begin{proof}
We will work in the universal cover. Denote by $\wF:T^1N\to\R$ the lift of $F$. Fix a base point $o\in N$ which belongs to a fundamental domain for the action of $\pi_1(M)$ denoted by $\Pi\dans N$. It is enough to prove that the integral of $\wF-P(F)$ on any (directed) geodesic ray starting from $\Pi$ whose length is large enough is negative.

Consider a directed geodesic ray $c$ starting from $\Pi$. Its ending point belongs to a fundamental domain $\gamma\Pi$ for some $\gamma\in\pi_1(M)$. Since the fundamental domains have finite diameters, a convexity argument shows that the geodesic ray $[o,\gamma o]$ shadows $c$. Since moreover $F$ is H{\"o}lder continuous and uniformly bounded, the difference of the integrals of $\wF-P(F)$ on $c$ and on $[o,\gamma o]$ is bounded independently of the length of $c$ and of $\gamma$ by a constant denoted by $K>0$.

Hence it is enough to prove that when the length of the segment $[o,\gamma o]$ is large enough the integrals $\int_o^{\gamma o}(\wF-P(F))$ are less that a uniform negative constant whose absolute value can be made arbitrarily big.

But in \cite{L} Ledrappier studied the Poincar{\'e} series
$$\sum_{\gamma\in\pi_1(M)}\exp\left[s\int_o^{\gamma o}(\wF-P(F))\right],$$
and proved that it has a finite critical exponent. In particular this series converges for some constant $s>0$ (his proof is dynamical and relies on Bowen's specification property \cite{B} as well as on the characterization \eqref{pressureperiodic} of the pressure). That means in particular that there exists a positive number $L_0>0$ such that when the length of $[o,\gamma o]$ is greater than $L_0$ the integral on $[o,\gamma o]$ of $\wF-P(F)$ is, for example, less than $-2K$. This completes the proof.
\end{proof}

\subsection{$F$-harmonic functions}

\paragraph{Gibbs kernel and Ledrappier's measures.} Let $F:T^1M\to\R$ be a H{\"o}lder continuous potential and $\wF:T^1N\to\R$ be its lift. We define the \emph{Gibbs kernel} associated to $F$ as the following function of $(y,z,\xi)\in N\times N\times N(\infty)$

\begin{equation}
\label{Eq:Fcocycle}
k^F(y,z;\xi)=\exp\left[\int_{\xi}^{z}\widetilde{F}-\int_{\xi}^{y}\widetilde{F}\right]\exp\left[-P(F)\beta_{\xi}(y,z)\right],
\end{equation}
where the difference of the two integrals stands for the limit as $T$ goes to infinity of $\int_{c(T)}^{z}\wF-\int_{c(T)}^{y}\wF$, where $c$ is \emph{any} geodesic ray asymptotic to $\xi$ and the integrals are taken on the \emph{directed} geodesic going from $c(T)$ to $y$ and $z$. This limit exists because $F$ is H{\"o}lder continuous: we can use the usual distortion controls. We clearly have the following cocycle relation

\begin{equation}
\label{Eq:Fcocyclerelation}
k^F(x,y;\xi)k^F(y,z;\xi)=k^F(x,z;\xi),
\end{equation}
when $x,y,z\in N$, and $\xi\in N(\infty)$. Moreover we have, by $\pi_1(M)$-invariance of $\wF$ and the invariance formula \eqref{Eq:equivariancebusemann}, the following invariance relation

\begin{equation}
\label{Eq:Fequivariance}
k^F(\gamma y,\gamma z;\gamma\xi)=k^F(y,z;\xi),
\end{equation}
which holds for $y,z\in N$, $\xi\in N(\infty)$ and $\gamma\in\pi_1(M)$.

The Gibbs kernel is normalized in the sense of \cite{L}. Ledrappier used in \cite{L} a Patterson-Sullivan argument (using the Poincar{\'e} series seen in the proof of Lemma \ref{dynamicalingredient}) in order to describe Gibbs states for the geodesic flow of a closed manifold with negative curvature (see also \cite{K} for the relation between invariant measures for the geodesic flows and measures on the boundary). 

\begin{theorem}[Ledrappier]
\label{ledrappiermeasures}
Let $M$ be a closed Riemannian manifold with negative sectional curvature, whose universal Riemannian cover is denoted by $N$. Then, there exists a unique (up to a multiplicative constant) family $(\nu^F_z)_{z\in N}$ of finite measures on $N(\infty)$ satisfying the following properties:
\begin{enumerate}
\item they are all in the same measure class;
\item the equivariance property $\gamma_{\ast}\nu^F_z=\nu^F_{\gamma z}$ for $\gamma\in\pi_1(M)$ and $z\in N$;
\item the cocycle property:
\begin{equation}
\label{ledrappiercocycle}
\frac{d\nu_z^F}{d\nu_y^F}(\xi)=k^F(y,z;\xi).
\end{equation}
\end{enumerate}
\end{theorem}

\paragraph{Remark 1.} The family of Ledrappier's measures is of course closely related to Gibbs states. With our convention for the kernel $k^F$ the typical trajectories of the geodesic flow for the \emph{unique} Gibbs state $\mu_F$ (see \cite{BR}) have the following description. Their lifts to $T^1N$ have:
\begin{itemize}
\item their past extremities which are distributed according to $\nu_o^F$;
\item their future extremities which are distributed according to $\nu_o^{\check{F}}$, where $\check{F}=F\circ\iota$, $\iota:T^1N\to T^1N$ being the involution $v\mapsto -v$.
\end{itemize}
Of course one can say more: lifts of Gibbs states to $T^1N$ have a very nice description in terms of Ledrappier's measures in the so called \emph{Hopf's coordinates}. We won't need this description here and we refer to \cite{C,PPS} for more details. This is the reason why we called the cocycle $k^F$ the Gibbs kernel.

\paragraph{$F$-harmonic functions.} As mentioned in the introduction, the Gibbs kernel shall play the role of the Poisson kernel. In the sequel we are interested in a class of continuous functions which have an integral representation similar to the Poisson representation of harmonic functions.

\begin{defi}
\label{Fharmonicfunction}
Let $M$ be a closed Riemannian manifold with negative sectional curvature and $N$ be its Riemannian universal cover. Let $F:T^1M\to\R$ be a H{\"o}lder continuous potential and let us fix a point $o\in N$. A positive function $h:N\to\R$ will be said to be $F$-harmonic if there exists a finite measure $\eta_o$ on $N(\infty)$ such that for any $z\in N$
$$h(z)=\int_{N(\infty)}k^F(o,z;\xi)\,d\eta_o(\xi).$$
The projection of an $F$-harmonic function to a quotient of $N$ will still be called $F$-harmonic.
\end{defi}

\paragraph{Remark 2.} This definition does not depend on the point $o$, because of the cocycle relation \eqref{Eq:Fcocycle}. The definition will work with another point $o'\in N$ if we state $\eta_{o'}(\xi)=k_F(o,o';\xi)\eta_o(\xi)$.

\paragraph{Remark 3.} Consider the function of $z\in N$ defined as
$$h_0^F(z)=\mass(\nu^F_z)=\int_{N(\infty)}k^F(o,z;\xi)d\nu^F_o(\xi).$$

It is $F$-harmonic and it descends to the quotient (thanks to equivariance property of the kernel \eqref{Eq:Fequivariance}). Consequently, the quotient function is $F$-harmonic on $M$ and will be denoted by $h_0^F$. We will prove later that, up to a multiplicative constant, this is the unique $F$-harmonic function on $M$.

\paragraph{Three examples.} We will give three important examples of Gibbs kernels:
\begin{itemize}
\item when $F=0$, the associated Gibbs kernel is given by the following
$$k^0(y,z;\xi)=\exp\left[-h\beta_{\xi}(y,z)\right],$$
$y,z\in N$, $\xi\in N(\infty)$, $h$ is the topological entropy of $g_t$ and the corresponding Ledrappier's measures $(\nu^0_z)_{z\in N}$ lie in the \emph{Patterson-Sullivan class};
\item the Gibbs kernel associated to the potential $\phi^u(v)=-d/dt|_{t=0}\log\,\Jac^u g_t(v),$ is given by
$$k^u(y,z;\xi)=\lim_{T\to\infty}\frac{\Jac^u\,G_{-T-\beta_{\xi}(y,z)}(v_{\xi,z})}{\Jac^u\,G_{-T}(v_{\xi,y})},$$
and the corresponding Ledrappier's measures $(\nu^{\phi^u}_z)_{z\in N}$ lie in the \emph{visibility class};
\item the Gibbs kernel associated to the potential $H(v)=d/dt|_{t=0}\log k(c_v(0),c_v(t),c_v(-\infty)),$ where $k$ is the Poisson kernel of $N$ (see \cite{AS}) coincides with the Poisson kernel. The corresponding Ledrappier's measures lie in the \emph{harmonic class} (see \cite{Su2})
\end{itemize}

\section{Shadows and Borel density}
\label{shadowscovborel}

In order to prove our Main Theorem, we will need an analogue of the usual theory of differentiation of measures in $\R^n$. Classical theorems, such as the Borel density theorem, strongly depend on the shapes of open sets that we choose to shrink into points: we choose open balls since we know that they satisfy Besicovitch's covering theorem.

In our framework, the sphere at infinity is no longer a smooth manifold, it only possesses a natural H{\"o}lder structure: the usual theorems don't apply directly. We first have to choose a family of open sets which will play the role of open balls. Shadows provide good candidates. These sets have been introduced by Sullivan \cite{Su1} in his study of conformal densities for the action at infinity of kleinian groups. They have proven to form a very useful tool in geometric ergodic theory (see for example \cite{C,K,Mo,PPS,R1,R2}).

\subsection{Definition}
\label{defshadows}

\paragraph{Shadows.} When $y,z\in N$ and $R>0$, we call \emph{shadow} seen from $y$ of the ball $B(z,R)$ the set denoted by $\OO_R(y,z)$, constituted by the points $\xi\in N(\infty)$ that are extremities of geodesic rays starting at $y$ passing through the open ball $B(z,R)$.

By definition of the cone topology on $N\cup N(\infty)$, these shadows are open sets of $N(\infty)$. Moreover the following lemma shows that they generate the natural topology of $N(\infty)$ (we state it without proof: see for example \cite{K} for a related discussion).
\begin{lemma}
\label{convergencedombres}
Assume that $R>0$ and $o\in N$ are fixed. Let $\xi\in N(\infty)$, and $(z_i)_{i\in\N}$ be a sequence which converges to $\xi$ in the cone topology. Then the sequence of shadows $\OO_R(o,z_i)$ converges to $\{\xi\}$.
\end{lemma}

\paragraph{Shadows seen from the boundary.} We can generalize the definition of shadows by defining for $R>0$, $z\in N$, and $\xi_0\in N(\infty)$, the set denoted by $\OO_R(\xi_0,z)$ that is constituted by points $\xi\in N(\infty)$ such that the geodesic $(\xi_0,\xi)$ meets the ball $B(z,R)$. We state the following lemma without proof (which is a rather immediate convexity argument) that will prove to be useful in the sequel.

\begin{lemma}
\label{convexiteombres}
Let $R>0$, $z\in N$, and $\xi_0\in N(\infty)$. Consider the cone $\CC_R(\xi_0,z)$ defined as the union of geodesics starting at $\xi_0$ that meet $B(z,R)$. Then:
\begin{enumerate}
\item the closed ball $\Cl\,B(z,R)$ separates $\Cl\,\CC_R(\xi_0,z)$ into two connected components, one of them containing $\xi_0$ ($\Cl$ stands for the closure in $N\cup N(\infty)$);
\item for every $z_0$ inside the connected component of $\Cl\,\CC_R(\xi_0,z)\moins\Cl\,B(z,R)$ containing $\xi_0$, we have $\OO_R(\xi_0,z)\dans\OO_R(z_0,z)$.
\end{enumerate}
\end{lemma}

\subsection{Shadow lemma}

The following result has already proven to be very useful. It has first been proven for conformal densities in constant curvature by \cite{Su1} and generalized in a more general setting by \cite{C,Mo,R1,PPS}.

Recall that $k^F$  denotes the Gibbs kernel and $(\nu_z^F)_{z\in N(\infty)}$ the family of Ledrappier's measures on the sphere at infinity. They are related by the following formula: $d\nu_z^F/d\nu_o^F(\xi)=k^F(o,z;\xi)$ for all $o,z\in N$, $\xi\in N(\infty)$.

\begin{theorem}[Shadow lemma]
\label{lemmedelombre}
Let $N$ be the universal cover of a closed and negatively curved Riemannian manifold $M$. Let $F:T^1M\to\R$ be a H{\"o}lder continuous potential. Let $R_0$ be the constant given in Lemma \ref{dynamicalingredient}. Fix a base point $o\in N$. Then there are two positive constants $C$ and $R>2R_0$ such that for all $z\in N$ and $\xi\in \OO_R(o,z)$, we have
$$\frac{C^{-1}}{\nu_o^F(\OO_R(o,z))}\leq k^F(o,z;\xi)\leq\frac{C}{\nu_o^F(\OO_R(o,z))}.$$
\end{theorem}

\paragraph{Remark 1.} Theorem \ref{lemmedelombre} is usually stated under the following (equivalent) form
$$C^{-1}k^F(z,o;\xi)\leq\nu^F_o(\OO_R(o,z))\leq Ck^F(z,o;\xi).$$
From the proof of this theorem (we refer to the author's PhD thesis (\cite{Al4},Th{\'e}or{\`e}me 6.3.5) for a proof of the present statement) we see that such an upper bound is valid no matter the choice of $R$: for every $R'>0$, there exists a constant $C(R')>0$ such that for every $z\in N$ and $\xi\in\OO_{R'}(o,z)$
$$k^F(o,z;\xi)\leq\frac{C(R')}{\nu_o^F(\OO_{R'}(o,z))}.$$

\paragraph{Remark 2.} The upper bound in the Shadow lemma implies that \emph{Ledrappier's measures are non-atomic}. Indeed, let $\xi\in N(\infty)$ and $o_t$ be the point of the geodesic ray $[o,\xi)$ being at distance $t$ to $o$. We have 
$$\nu^F_o(\OO_R(o,o_t))\leq C k^F(o_t,o;\xi)=C\exp\left(\int_o^{o_t}(\widetilde{F}-P(F)\right).$$
Argueing as in the proof of Lemma \ref{dynamicalingredient}, which strongly relies on the compactness of $M$, we see that the right hand side tends to $0$ as $t$ goes to infinity. Since these shadows form a decreasing system of neighbourhoods of $\xi$, this implies that $\nu_o^F(\{\xi\})=0$.

\paragraph{Remark 3.} We will need later the condition $R>2R_0$, which is not usually stated. In classical references the lemma provides the existence of an $R$ such that the lower bound holds, but the proof shows in reality that such an $R$ can be chosen arbitrarily large. Indeed the classical proof consists in showing that an $R$ is admissible if and only if the measures $\nu_z^F(\OO_R(o,z))$ are bounded from above by a constant independent of $z$. In order to find such an $R$ we argue by contradiction assuming the existence of a sequence of positive numbers $(R_i)_{i\in\N}$ going to infinity and a sequence $(z_i)_{i\in\N}$ of elements of $N$ such that $\nu_{z_i}^F(\OO_{R_i}(o,z_i))\to 0$ as $i\to\infty$. If the latter were true, it would imply the existence of an atom in the class of Ledrappier's measures which is absurd (see the previous remark). This argument provides more than just an admissible $R$: it provides arbitrarily large ones.

\paragraph{Remark 4.} It comes from the Shadow lemma that Ledrappier's measures charge every nonempty open subset of $N(\infty)$.

\subsection{Borel density}

Borel's differentiation theorem usually stands for Radon measures on Euclidean (and Riemannian) spaces, using Euclidean (Riemannian) balls: see \cite{Matt}. We will need an analogue for Radon measures of the sphere at infinity $N(\infty)$.

\begin{theorem}[Borel's differentiation theorem]
\label{boreldifferentiation}
Let $N$ be the universal cover of a closed and negatively curved Riemannian manifold $M$. Assume that to every $\xi\in N(\infty)$ is associated a sequence of shadows $\OO_i(\xi)=\OO_R(o,z_i)$, where $z_i$ is a sequence which converges nontangentially to $\xi$ and $R$ is the positive number given by the shadow lemma. Let $\eta_1,\eta_2$ two finite Radon measures on $N(\infty)$. Write the Lebesgue decomposition of $\eta_1$ with respect to $\eta_2$ as
$$\eta_1=f\eta_2+\eta_s,$$
where $f$ is $\eta_2$-integrable on $N(\infty)$, and $\eta_s$ is singular with respect to $\eta_2$.

Then, for $\eta_2$-almost every $\xi\in N(\infty)$, we have
$$\lim_{i\to\infty}\frac{\eta_1(\OO_i(\xi))}{\eta_2(\OO_i(\xi))}=f(\xi).$$

In particular, when $\eta_1$ and $\eta_2$ are singular, we have for $\eta_1$-almost every $\xi\in N(\infty)$
$$\lim_{i\to\infty}\frac{\eta_1(\OO_i(\xi))}{\eta_2(\OO_i(\xi))}=\infty.$$
\end{theorem}

Following an indication of Ledrappier we found a intricate argument using the theory of Markov partitions. We learned from a recent preprint of Tanaka \cite{T} a very nice argument which uses the existence of bilipschitz embeddings of the sphere at infinity in Euclidean spaces. Thiis argument is very general and works for the boundary of more general Gromov hyperbolic spaces. We give Tanaka's argument in Appendix.

\section{Reduction to the study of the maximal function}
\label{redmax}

In all this section, $M$ represents a closed and negatively curved manifold and $N$ denotes its Riemannian universal cover. We consider a H{\"o}lder continuous potential $F:T^1M\to\R$ as well as its lift $\wF:T^1\to\R$.

The strategy is to use a maximal function \emph{{\`a} la Hardy-Littlewood} and to prove that up to a universal constant it bounds the quotient of any $F$-harmonic functions. Then we deduce a special case of Theorem \ref{theoremedefatou}, namely when the denominator is $h_0^F$ from which we deduce the general case.

\subsection{The key inequality}
\paragraph{Maximal function.} Consider two finite Radon measures on $N(\infty)$, $\eta_1,\eta_2$. The \emph{maximal function} of $\eta_1$ with respect to $\eta_2$ is given by the following formula

\begin{equation}
\label{maximal}
M_{\eta_1/\eta_2}(\xi)=\Sup_{z\in [o,\xi)}\frac{\eta_1(\OO_R(o,z))}{\eta_2(\OO_R(o,z))},
\end{equation}
where $\xi\in N(\infty)$.

The important property of the maximal function is that it is weakly $\eta_2$-integrable. This fact also follows from Tanaka's argument using the existence of bilipschitz  embeddings of $N(\infty)$ endowed with natural visual metrics, inside a Euclidean space (this is Bonk-Schramm's Theorem \cite{BS}). The proof shall be postponed until the appendix. 

\begin{lemma}
\label{maxfaiblementl1}
There exists a constant $A>0$, such that for every finite Radon measures on $N(\infty)$, $\eta_1,\eta_2$, and every positive number $\alpha>0$, we have
$$\eta_2[M_{\eta_1/\eta_2}>\alpha]\leq \frac{A}{\alpha}\mass(\eta_1).$$
\end{lemma}

\paragraph{The key proposition.} The following proposition is the main technical ingredient of the proof of Theorem \ref{theoremedefatou}: we shall prove it later in Section \ref{proofkey}. Define $\V^r(o,\xi)$ as the $r$-neighbourhood of the geodesic ray $[o,\xi)$.

\begin{proposition}
\label{ingredientfatou}
Let $\eta_1,\eta_2$ be two finite Radon measures on $N(\infty)$ and $h_1,h_2$ the corresponding $F$-harmonic functions.
\begin{enumerate}
\item Then for every $r>0$, there exists a constant $C_r>0$, such that for every $\xi\in N(\infty)$ and $z\in \V^r(o,\xi)$ we have
$$\frac{h_1(z)}{h_2(z)}\leq C_r M_{\eta_1/\eta_2}(\xi).$$
\item If $C>1$ is the constant given by the shadow lemma, we have for all $z\in N$ and $i=1,2$
$$h_i(z)\geq C^{-1}\frac{\eta_i(\OO_R(o,z))}{\nu_o^F(\OO_R(o,z))}.$$
\end{enumerate}
\end{proposition}

\subsection{Proof of Theorem \ref{theoremedefatou}}

We first show the theorem when $\eta_2=\nu_o^F$, i.e. when the function $h_2$ maps every $z$ on the mass of Ledrappier's measure $\nu_z^F$.

\begin{proposition}
\label{specialcase}
Let $\eta=f\nu_o^F+\eta_s$ be a finite Radon measure on $N(\infty)$ where $f$ is $\nu_o^F$-integrable and $\eta_s$ is singular with respect to $\nu_o^F$. Let $h$ be the corresponding $F$-harmonic function. Then:
\begin{enumerate}
\item for $\nu_o^F$-almost every $\xi\in N(\infty)$, we have
$$\frac{h(z)}{h_0^F(z)}\to f(\xi),$$
as $z$ converges nontangentially to $\xi$;
\item for $\eta_s$-almost every $\xi\in N(\infty)$, we have
$$h(z)\to\infty$$
as $z$ converges nontangentially to $\xi$.
\end{enumerate}
\end{proposition}

The following result is an immediate corollary of Proposition \ref{specialcase}: it uses the fact the $h_0^F$ is uniformly bounded away from zero and infinity.

\begin{coro}
\label{liminf}
Let $\eta$ be a finite Radon measure on $N(\infty)$ and $h$ be the corresponding $F$-harmonic function. Then for $\eta$-almost every $\xi\in N(\infty)$ we have:
$$\limi_{z\to\xi} h(z)>0,$$
where the convergence is nontangential.
\end{coro}

\paragraph{Continuous density.} We first prove Proposition \ref{specialcase} when $\eta_1=f\nu_o^F$, with $f:N(\infty)\to\R^+$ continuous.

\begin{lemma}
\label{densitecontinue}
Let $f:N(\infty)\to\R^+$ be a continuous function and $\eta=f\nu_o^F$. Consider the corresponding $F$-harmonic function $h:N\to\R$.

Then for $\nu_o^F$-almost every $\xi\in N(\infty)$, we have
$$\lim_{z\to\xi}\frac{h(z)}{h_0^F(z)}=f(\xi),$$
where the convergence is nontangential.
\end{lemma}

\begin{proof}
As a preliminary remark, let us mention that since $k^F(o,z;\xi)=d\nu_z^F/d\nu_o^F(\xi)$, we have $h(z)=\int_{N(\infty)}f(\xi)d\nu_z^F(\xi)$ for all $z\in N$.

Assume first that $(z_i)_{i\in\N}$ converges to $\xi_0$ while staying on the geodesic ray $[o,\xi_0)$, and that the distance of $z_i$ to $o$ is increasing with $i$. The shadows $\OO_i(\xi_0)=\OO_R(o,z_i)$ then form a decreasing sequence of neighbourhoods of $\xi_0$ which converges to $\{\xi_0\}$.

Let $\eps>0$ and $i_0$ be such that for all $i\geq i_0$, we have $|f(\xi)-f(\xi_0)|<\eps$ for all $\xi\in\OO_i(\xi_0)$. Then for all $z\in N$, we can write
$$\frac{h(z)}{h_0^F(z)}-f(\xi_0)=\int_{^c\OO_{i_0}(\xi_0)} (f(\xi)-f(\xi_0)) \frac{d\nu_z^F(\xi)}{\mass(\nu_z^F)}+\int_{\OO_{i_0}(\xi_0)}(f(\xi)-f(\xi_0)) \frac{d\nu_z^F(\xi)}{\mass(\nu_z^F)}.$$

If we apply this formula to $z_i$ with $i\geq i_0$, we see that the second term is controlled by $\eps$.

As for the first one we bound it from above by $(\nu^F_{z_i}(^c\OO_{i_0}(\xi_0))/\mass(\nu_{z_i}^F))2\Sup f$. In order to conclude we only have to prove that $\lim\nu^F_{z_i}(^c\OO_{i_0}(\xi_0))=0$. 

Since $z_i\to\xi_0$ while staying on $[o,\xi_0)$, the sequence of shadows $\OO_R(\xi_0,z_i)$ converges to $N(\infty)\moins\{\xi\}$. In particular for some $i_1>i_0$, we have $^c\OO_{i_0}(\xi_0)\dans\OO_R(\xi_0,z_{i_1})$. 

By Lemma \ref{convexiteombres}, for every $i>i_1$ we have $\OO_R(\xi_0,z_{i_1})\dans\OO_R(z_i,z_{i_1})$. Thus for every $i>i_1$ such that $\dist(z_i,z_{i_1})>R$, we find $\nu_{z_i}^F(^c\OO_{i_0}(\xi_0))\leq\nu_{z_i}^F(\OO_R(z_i,z_{i_1}))$. Note that for this last argument, we don't need to ask that $z_i\in[o,\xi_0)$, but only that $z_i$ belongs to the component of $\Cl\,\CC_R(\xi_0,z)\moins\Cl\,B(z_{i_1},R)$ containing $\xi_0$.

Now it is possible to send $z_i$ on the \emph{compact} fundamental domain $\Pi$ which contains $o$ by an isometry $\gamma_i\in\pi_1(M)$. As $i$ increases, $\dist(z_i,z_{i_1})$ diverges to infinity, $\gamma_i$ is an isometry and $\gamma_i z_i\in\Pi$, which is a fixed compact fundamental domain. Therefore, $\gamma_i\,z_{i_1}$ is sent at infinity and the diameter of the shadow $\OO_R(\gamma_i z_i,\gamma_i z_{i_1})$ goes to zero. Consequently since $\nu_{\gamma_i z_i}^F$ has no atom
$$\nu_{z_i}^F(^c\OO_{i_0}(\xi_0)))\leq\nu_{z_i}^F(\OO_R(z_i,z_{i_1}))=\nu_{\gamma_i z_i}^F(\OO_R(\gamma_i z_i,\gamma_i z_{i_1}))\To_{i\to\infty}0.$$

Finally, there is $i_2>i_1>i_0$ such that for every $i>i_2$, $\nu_{z_i}^F(^c\OO_{i_0}(\xi_0)))\leq\eps$. Hence, when $i>i_2$
$$\left|\frac{h(z_i)}{\widetilde{h}_0(z_i)}-f(\xi_0)\right|\leq 2\Sup f\eps+\eps=(2\Sup f+1)\eps.$$

This proves that $\lim h(z_i)/h(z_i)=f(\xi_0)$.

Now, it is easy to deduce from the proof above that the conclusion holds true if we ask that the points $z_i$ stay at distance $\leq R$ to $[o,\xi_0)$. Indeed in that case there exists a sequence of $z_i'\in [o,\xi_0)$ such that we have $\dist(z_i,z_i')\leq R$ for every $i$. We can consider the sequence of neighbourhoods $\OO_i(\xi_0)=\OO_R(o,z_i')$, and use the fact that the last part of the argument works also for $z_i\in B(z_i',R)$. Consider an arbitrarily large parameter $R$ to conclude the proof.
\end{proof}

\paragraph{Integrable density.} Now we treat the case of an integrable density. In order to do this we approximate integrable functions by continuous ones in the $L^1$-topology and use the key property \ref{ingredientfatou} as well as the weak integrability of the maximal function: the proof is just a copy of that of \cite[Theorem 7.7]{Ru}.

\begin{lemma}
\label{densiteintegrable}
Let $f:N(\infty)\to\R^+$ be a $\nu_o^F$-integrable function and $\eta=f\nu_o^F$. Consider the corresponding $F$-harmonic function $h:N\to\R$.

Then for $\nu_o^F$-almost every $\xi\in N(\infty)$, we have
$$\lim_{z\to\xi}\frac{h(z)}{h_0^F(z)}=f(\xi),$$
where the convergence is nontangential.
\end{lemma}

\paragraph{Singular measure.} Finally we treat the case of a $F$-harmonic function corresponding to a measure which is singular with respect to $\nu_o^F$. The proof of the following lemma uses the Borel's differentiation theorem \ref{boreldifferentiation} as well as the two inequalities of Proposition \ref{ingredientfatou}. The proof may be copied from Rudin's proof of \cite[Theorem 11.22]{Ru}. This lemma allows us to complete the proof of Proposition \ref{specialcase}.

\begin{lemma}
\label{singular}
Let $\eta_s$ be a measure which is singular with respect to $\nu_o^F$ and $h$ be the corresponding $F$-harmonic function. Then:
\begin{enumerate}
\item for $\nu_o^F$-almost every $\xi\in N(\infty)$, we have $h(z)\to 0$ as $z$ converges nontangentially to $\xi$;
\item for $\eta_s$-almost every $\xi\in N(\infty)$, we have $h(z)\to \infty$ as $z$ converges nontangentially to $\xi$.
\end{enumerate}
\end{lemma}

\paragraph{Proof of Theorem \ref{theoremedefatou}}

The proof of the general case of Theorem \ref{theoremedefatou} follows the same steps as that of the particular case where $\eta_2=\nu_o^F$. The three lemmas \ref{densitecontinue}, \ref{densiteintegrable} and \ref{singular} can be stated with $\eta_2$ instead of $\nu_o^F$. The only difference in the proof is that we have to use Corollary \ref{liminf} in order to get that quotients of the form $o(1)/h_2(z)$ are in reality $o(1)$ when $z$ converges nontangentially to $\xi$ which is $\eta_2$-typical, and that quotients of the form $h_1(z)/o(1)$ diverge to infinity as $z$ converges nontangentially to $\xi$ which is $\eta_1$-typical.\quad \hfill $\square$

\section{Proof of the key proposition}
\label{proofkey}

\subsection{Proof of the second part}

The second part of Proposition \ref{ingredientfatou} is the easiest one to prove: it is a simple application of the shadow lemma. We want to prove that for every finite Radon measure $\eta$ on $N(\infty)$, if $h$ is the corresponding $F$-harmonic function, we have for every $z\in N$
$$h(z)\geq C^{-1}\frac{\eta(\OO_R(o,z))}{\nu_o^F(\OO_R(o,z))}.$$

So let $\eta$ be a finite Radon measure on $N(\infty)$, and $h$ the corresponding $F$-harmonic function. Given $z\in N$, we have $h(z)\geq\int_{\OO_R(o,z)}k^F(o,z;\xi)d\eta(\xi)$.

Now use the lower bound given by the shadow lemma. There is a constant $C$ independent of $z$ such that for every $\xi\in\OO_R(o,z)$, we have $k^F(o,z;\xi)\geq C^{-1}/\nu_o^F(\OO_R(o,z))$.

We conclude the proof by injecting this inequality in the integral.\quad \hfill $\square$

\subsection{Proof of the first part}

\paragraph{Reduction to the radial case.} We start by showing that it is enough to prove the first assertion of Proposition \ref{ingredientfatou} when $r=0$, i.e. when $z$ is supposed to stay on a given geodesic ray $[o,\xi)$. For that purpose we show a version of Harnack's inequality for $F$-harmonic functions.
\begin{proposition}
\label{Harnack}
Let $r>0$. Then there exists a constant $A_r>0$ such that for every function $h:N\to (0,\infty)$ which is $F$-harmonic function and every couple $y,z\in N$ which are distant of at most $r$ we have
$$A_r^{-1}\leq\frac{h(z)}{h(y)}\leq A_r.$$
\end{proposition}

Hence if it is true that for every $\xi\in N(\infty)$ and $z\in [o,\xi)$ the inequality $h_1(z)/h_2(z)\leq C_0M_{\eta_1/\eta_2}(\xi)$ holds, for measures $\eta_i$ and corresponding $F$-harmonic functions $h_i$, then for all $y\in\V^r(o,\xi)$, there exists $z\in [o,\xi)$ such that $\dist(y,z)\leq r$, and $h_1(y)/h_2(y)\leq A_r^2 h_1(z)/h_2(z)\leq A_r^2 C_0 M_{\eta_1/\eta_2}(\xi)$. Hence we are reduced to the radial case. Proposition \ref{Harnack} is an immediate consequence of Lemma \ref{distortioncontrol} below, which is a distortion control, and of the following inequality which holds true for every $h$ $F$-harmonic corresponding to $\eta$, and $y,z\in N$
$$\frac{h(z)}{h(y)}=\frac{\int_{N(\infty)}k^F(o,z;\xi)d\eta(\xi)}{\int_{N(\infty)}k^F(o,y;\xi)d\eta(\xi)}\leq\Sup_{\xi\in N(\infty)}k^F(y,z;\xi).$$

\begin{lemma}
\label{distortioncontrol}
Let $r>0$. There exists a number $L_r>0$ such that whenever $y,z\in N$ are distant of at most $r$ we have for every $\xi\in N(\infty)$
$$L_r^{-1}\leq k^F(y,z;\xi)\leq L_r.$$
\end{lemma}

\begin{proof}
Consider a real number $r>0$ as well as two points $y,z\in N$ that are distant of at most $r$. Let $\xi\in N(\infty)$. It is enough to prove the  upper bound since $k^F(y,z;\xi)=(k^F(z,y;\xi))^{-1}$.

With no restriction we may assume that $\beta_{\xi}(y,z)$ is positive, in such a way that there exists $z'\in [z;\xi)$ belonging to the horosphere centered at $\xi$ and passing through $y$. Then we may write
$$k^F(y,z;\xi)=\exp\left[\int_{\xi}^{z'}\wF-\int_{\xi}^{y}\wF\right]\exp\left[\int_{z'}^z(\wF-P(F))\right].$$

We have $\dist(z',z)=\beta_{\xi}(y,z)\leq\dist(y,z)\leq r$, thus by triangular inequality, $\dist(y,z')\leq\dist(y,z)+\dist(z,z')\leq 2r$. By comparison of geodesic and horospheric distances (see Theorem 4.6 of \cite{HI}) we see that $\dist_H(z,z'')\leq 2/b\sinh(2br)$ (here $\dist_H$ denotes the horospheric distance).

By using the H{\"o}lder continuity of $F$ as well as the exponential contraction of horospheric distances along the geodesic rays $[z',\xi)$ et $[y,\xi)$, we can by the usual distortion controls bound from above the first factor by a number depending only on $r$ and on $F$.

By using the fact that $\wF-P(F)$ is bounded and that $\dist(z',z)\leq r$, the second term can also be bounded by a term depending only on $r$ and $F$. It is then possible to conclude the proof.
\end{proof}

\paragraph{Two geometric lemmas.} The two following lemmas are rather immediate consequences of usual comparison theorems and are the main geometric ingredients of the proof of Proposition \ref{ingredientfatou}. The theorem of Toponogov \cite{CE} allows us to compare triangles of $N$ with triangles of $N_{-b^2}$. The CAT($-a^2$) inequalities \cite{GdH} allow us to compare triangles of $N$ with triangles of $N_{-a^2}$.

\begin{lemma}
\label{lemmegeometrique1}
There exists a positive number $\theta_0$ such that for every $z\in N$, and every $\xi\in\,^c\OO_{R/2}(o,z)$, we have
$$\widehat{\xi z\xi_0}\geq\theta_0,$$
where $\xi_0$ is the extremity of the geodesic ray starting from $o$ and passing through $z$.
\end{lemma}

\begin{proof}
We will restrict our study to that of a simpler case. Let $\xi_1$ be the other extremity of the geodesic passing through $o$ and $z$, in such a way that $\xi_1,o,z,\xi_0$ are aligned in this order on the geodesic $(\xi_1,\xi_0)$. Lemma \ref{convexiteombres} asserts that $\OO_{R/2}(\xi_1,z)\dans\OO_{R/2}(o,z)$. Thus it is enough to bound from below the exterior angle at $z$ of the geodesic triangle whose vertices are $\xi_1$, $z$ and $\xi$ where $\xi\in\,^c\OO_{R/2}(\xi_1,z)$.

It is immediate that it is enough to bound this angle from below when the geodesic $(\xi_1,\xi)$ is tangent to the sphere $S(z,R/2)$, i.e. when the altitude of the triangle is $R/2$. Consider such a triangle and denote by $\theta$ the exterior angle at $z$.

Toponogov's theorem implies that if one considers the triangle, denoted be $\Delta_{R/2}$ (unique up to isometry in constant curvature) of $N_{-b^2}$ with two vertices at infinity, a vertex denoted by $\overline{z}$ in $N_{-b^2}$ and an altitude of $R/2$, then the exterior angle $\theta_0$ at $\overline{z}$ is $\leq\theta$.

Finally an application of Gauss-Bonnet's theorem gives the explicit lower bound:
$$\theta\geq\theta_0=b^2\Area(\Delta_{R/2}).$$
\end{proof}

\begin{lemma}
\label{lemmegeometrique2}
There exists positive numbers $K_1,K_2>0$ such that for every $y,z\in N$ such that $o,y,z$ are aligned in this order and for every $\xi\in\,^c\OO_{R/2}(o,y)$, we have:
\begin{enumerate}
\item the $\xi$-horospheric distance between $y$ and the geodesic $(z,\xi)$ is less than $K_1$;
\item $\beta_{\xi}(y,z)\geq\dist(y,z)-K_2$.
\end{enumerate}
\end{lemma}

\begin{proof}
We first claim that for every $K>0$, there exists a number $K'>0$ (depending only on the curvature and on $K$) such that when $x\in N$ is at distance $\leq K$ of a directed geodesic $c$, then $\dist_H(x,c)\leq K'$ (let us say that the horospheric projection is with respect to horospheres passing through $c(\infty)$). Let $x_h\in c$ (resp. $x_g$) be the horospheric (resp. geodesic) projection of $x$ on $c$. Then by Theorem 4.6 of \cite{HI} we have $\dist_H(x,x_h)\leq 2/b\sinh(b\dist(x,x_h))$. Now let $\beta=\dist(x_g,x_h)$ (this is the Busemann cocycle at $c(\infty)$) in such a way that by triangular inequality, $\dist(x,x_h)\leq \dist(x,x_g)+\beta\leq K+\beta$.

Let us use Proposition 4.4 of \cite{HI} (which gives comparison triangles) to prove that $\beta$ is bounded by its analogue $\beta_{-b^2}$ in $N_{-b^2}$ (which is a multiple of its analogue $\beta_{-1}$ in $N_{-1}$). It is enough to treat the problem in the hyperbolic plane when $c$ is the vertical half line starting from $0$ and $x=e^{\mathbf{i}t}$, $t\in(0,\pi/2]$, whose distance to $c$ is $\leq K$ (in particular $t$ is uniformly bounded from below). We have $\beta_{-1}=-\log\,\sin(t)$ which is bounded above by a constant depending only on $K$. This concludes the proof of the claim.

Now, let $y,z$ and $\xi$ be given such as in the statement of the lemma. Lemma \ref{lemmegeometrique1} ensures that since $\xi\in\,^c\OO_{R/2}(o,y)$, and since $o,y,z$ are aligned in this order, if $\theta$ represents $\widehat{\xi yz}$, we have $\theta\geq\theta_0$. By the claim above, in order to bound the horospheric distance between $y$ and $(z,\xi)$, it is enough to bound uniformly the geodesic distance between them.

Denote by $\xi_0$ the extremity of the geodesic ray $[o,z)$. An immediate argument of convexity inside the triangle with vertices $\xi,y,\xi_0$ shows that any geodesic ray starting from y and passing through the geodesic $(\xi_0,\xi)$ has to cut the geodesic ray $[z,\xi)$. In particular it implies that $\dist(y,(z,\xi))\leq\dist(y,(\xi_0,\xi))$. Denote by $\delta$ the latter distance.

We now use the CAT($-a^2$) inequality to the triangle whose vertices are $y,\xi,\xi_0$, whose angle at $y$ is $\theta\geq\theta_0$. A comparison triangle in $N_{-a^2}$ is the unique triangle with two vertices at infinity and an angle $\theta\geq\theta_0$, that we call $\Delta_{\theta}$. Its altitude is less than that of the triangle $\Delta_{\theta_0}$. Hence a CAT($-a^2$) inequality (Crit\`ere C of \cite{GdH}) gives that $\delta$ is less than the altitude of $\Delta_{\theta_0}$, which is independent of $y,z,\xi$. This concludes the proof of the first assertion.

The second assertion comes directly from the triangular inequality applied to the triangle with vertices $y,z$ and the $\xi$-horospheric projection of $y$ on the geodesic $(z,\xi)$.

\end{proof}

\paragraph{Decomposition of the integral representation.} Let $o\in N$ be any base point, $\xi_0\in N(\infty)$ and $z\in N$ such that $z\in[o,\xi_0)$.  Consider a subdivision of the segment $[o,z]$ of the form $(z_i)_{i=0,...,k(z)}$ where $o,z_{k(z)},...,z_1,z_0=z$ are aligned in this order and such that for all $i\geq 1$
$$\dist(z_{i-1},z_i)=\frac{R}{2}.$$
Note that we implicitly asked that $\dist(o,z_{k(z)})<R/2$ in such way that $\OO_R(o,z_{k(z)})=N(\infty)$. We will consider the decomposition of $N(\infty)$ given by the shadows $\OO_R(o,z_i)$.

\begin{lemma}
\label{inclusion}
For every $i\geq 1$, we have
$$\OO_{R/2}(o,z_i)\dans\OO_R(o,z_{i-1})\dans\OO_R(o,z_i).$$
\end{lemma}

\begin{proof}
The proof of this lemma is immediate: for all $i\geq 1$, $o, z_i, z_{i-1}$ are aligned in this order and we have $\dist(z_i,z_{i-1})=R/2$, in such a way that $B(z_i,R/2)\dans B(z_{i-1},R)$.
\end{proof}

Now consider a $F$-harmonic function $h:N\to (0,\infty)$ corresponding to a Borel measure $\eta$. Denote by $\OO_i$ the shadow $\OO_R(o,z_i)$. We can decompose the integral representation of $h$ as

\begin{equation}
\label{cutintegral}
h(z)=\sum_{i=1}^{k(z)}\int_{\OO_i\moins\OO_{i-1}}k^F(o,z;\xi)d\eta(\xi)+\int_{\OO_0} k^F(o,z;\xi)d\eta(\xi).
\end{equation}

The idea will be now to use the cocycle relation $k^F(o,z;\xi)=k^F(o,z_i;\xi)k^F(z_i,z;\xi)$ when $\xi\in\OO_i\moins{\OO_{i-1}}$, to use the shadow lemma in order to control the first factor, and a geometric estimate, which is the object of the next paragraph, in order to control the second one.

\paragraph{Estimates for the Gibbs kernel.} We shall use our geometric lemmas in order to prove the following bounds for the kernel:

\begin{lemma}
\label{estimatecocycle}
There exists a constant $C_0>1$ independent of $z$ for which we have
$$C_0^{-1}\exp\left[\int_{z_i}^z(\widetilde{F}-P(F))\right]\leq k^F(z_i,z;\xi)\leq C_0\exp\left[\int_{z_i}^z(\widetilde{F}-P(F))\right],$$
for every $i\in\{1,...,k(z)\}$, and $\xi\in\OO_i\moins\OO_{i-1}$.
\end{lemma}

\begin{proof}
Let $z\in N$, $i\in\{1,...,k(z)\}$, and $\xi\in\OO_i\moins\OO_{i-1}$. Recall that the interest of considering the point $z_i$ is that by Lemma \ref{lemmegeometrique1} we have that $\widehat{\xi z_i \xi_0}$ is more than a universal angle $\theta_0$.

We deduce from Lemma \ref{lemmegeometrique2} that for every $\xi\in\OO_i\moins\OO_{i-1}$ we have $\beta_{\xi}(z_i,z)\geq\dist(z_i,z)-K_2$. Thus we may assume that $\beta_{\xi}(z_i,z)>0$ because if not $\dist(z_i,z)\leq K_2$: the Busemann cocycle is negative only on a bounded subset of $[z_i,\xi)$ on which such estimates hold since every of the terms involved is bounded by a quantity depending only on $K$ and $F$ (see Lemma \ref{distortioncontrol}).

Suppose in the sequel that $\beta_{\xi}(z_i,z)>0$. Let $z'$ be the intersection point between the geodesic starting from $\xi$ and passing through $z$, with the horosphere $H_{\xi}(z_i)$. Since by hypothesis $\beta_{\xi}(z,z_i)>0$, the points $z,z'$ and $\xi$ are aligned in this order. Hence the kernel reads as follows
$$k^F(z_i,z;\xi)=\exp\left[\int_{\xi}^{z'}\widetilde{F}-\int_{\xi}^{z_i}\widetilde{F}\right]\exp\left[\int_{z'}^z(\widetilde{F}-P(F))\right].$$

By Lemma \ref{lemmegeometrique2}, $\dist_H(z_i,z')\leq K_1$ which is uniform. Thus, the first factor can be bounded exactly like in the proof of Lemma \ref{distortioncontrol}.

Now we have to treat the second factor. By convexity of the $K_1$-neighbourhood of $[z',z]$, the geodesic segment $[z_i,z]$ $K_1$-shadows $[z',z]$, and $K_1$ is a uniform constant. Hence since the function $F$ is H{\"o}lder on $T^1M$, it comes that the ratio between $\exp[\int_{z'}^z(\widetilde{F}-P(F))]$ and $\exp[\int_{z_i}^z(\widetilde{F}-P(F))]$ is uniformly bounded, concluding the proof of the lemma.
\end{proof}

\paragraph{Bounds for $F$-harmonic functions.} We will now use the previous lemmas in order to give our main estimate for $F$-harmonic functions.

\begin{lemma}
\label{decintrepr}
Let $\eta$ be a finite Radon measure on $N(\infty)$ and $h$ be the corresponding $F$-harmonic function. Consider $\xi_0\in N(\infty)$, $z\in [o,\xi_0)$ and $(z_i)_{i=0,...,k(z)}$ the corresponding subdivision of the segment $[o,z]$.  Set
$$k_i=\exp\left[\int_{z_i}^z(\widetilde{F}-P(F))\right].$$
Then
$$C_1^{-1}\left[\sum_{i=1}^{k(z)}\frac{k_i}{\nu_o^F(\OO_i)}\eta(\OO_i\moins\OO_{i-1})+\frac{\eta(\OO_0)}{\nu_o^F(\OO_0)}\right]\leq h(z)
\leq C_1\left[\sum_{i=1}^{k(z)}\frac{k_i}{\nu_o^F(\OO_i)}\eta(\OO_i\moins\OO_{i-1})+\frac{\eta(\OO_0)}{\nu_o^F(\OO_0)}\right],$$
for a constant $C_1$ independent of $\eta$, $z$.
\end{lemma}

\begin{proof}
When $\xi\in\OO_0$ we can use the shadow lemma in order to get $k^F(o,z;\xi)\in[C^{-1},C]1/\nu_o^F(\OO_0)$. Injecting this in the integral allows us to prove that

$$C^{-1}\frac{\eta(\OO_0)}{\nu_o^F(\OO_0)}\leq\int_{\OO_0} k^F(o,z;\xi)d\eta(\xi)\leq C\frac{\eta(\OO_0)}{\nu_o^F(\OO_0)}.$$

Now when $\xi\in\OO_i\moins\OO_{i-1}$ for $i\geq 1$ use the cocycle relation $k^F(o,z;\xi)=k^F(o,z_i;\xi)k^F(z_i,z;\xi)$.

Use the shadow lemma to show that $k^F(o,z_i;\xi)\in[C^{-1},C]1/\nu_o^F(\OO_i)$. Finally use Lemma \ref{estimatecocycle} in order to get $C_0^{-1}k_i\leq k^F(z_i,z;\xi)\leq C_0 k_i$ and inject these inequalities in the integral in order to get a constant $C_1$ such that
$$C_1^{-1}\frac{k_i}{\nu_o^F(\OO_i)}\eta(\OO_i\moins\OO_{i-1})\leq \int_{\OO_i\moins\OO_{i-1}}k^F(o,z;\xi)d\eta(\xi)\leq C_1\frac{k_i}{\nu_o^F(\OO_i)}\eta(\OO_i\moins\OO_{i-1}).$$

Summing over $i$ completes the proof.
\end{proof}

\paragraph{End of the proof of the key proposition.} Consider two finite Radon measures $\eta_1$ and $\eta_2$ as well as the corresponding $F$-harmonic functions $h_1$ and $h_2$.  Denote for $i=0,...,k(z)$ (note that $k_0=1$) $a_i=k_i/\nu_o^F(\OO_i).$

Performing an Abel transformation in the sum given by Lemma \ref{decintrepr} yields
\begin{eqnarray*}
h_1(z)&\leq& C_1\left[\sum_{i=1}^{k(z)}a_i(\eta_1(\OO_i)-\eta_1(\OO_{i-1}))+a_0\eta_1(\OO_0)\right]\\
      &\leq& C_1\left[\sum_{i=0}^{k(z)-1}(a_i-a_{i+1})\eta_1(\OO_i)+a_{k(z)}\eta_1(\OO_{k(z)})\right].
\end{eqnarray*}

Note first that by definition of the maximal function, we have
$$a_{k(z)}\eta_1(\OO_{k(z)})\leq a_{k(z)}\eta_2(\OO_{k(z)})\,M_{\eta_1/\eta_2}(\xi).$$

Note now that for every $i=i_0,...,k(z)-1$, we have
$$a_i-a_{i+1}=k_i\left(\frac{1}{\nu_o^F(\OO_i)}-\frac{k_{i+1}}{k_i}\frac{1}{\nu_o^F(\OO_{i+1})}\right).$$

On the one hand, $(\OO_i)_{i=0,...,k(z)}$ forms a strictly increasing sequence of shadows: in particular we have for every $i\geq 0$, $1/\nu_o^F(\OO_i)> 1/\nu_o^F(\OO_{i+1})$.

On the other hand, when $i\geq 0$ we have

$$\frac{k_{i+1}}{k_i}=\exp\left[\int_{z_{i+1}}^{z_i}(\wF-P(F))\right].$$

But we have chosen the subdivision in such a way that $\dist(z_{i+1},z_i)=R/2>R_0$, where $R_0$ is the constant given by the dynamical Lemma \ref{dynamicalingredient}. This implies in particular that $k_{i+1}/k_i<1$. All this implies that $a_i-a_{i+1}>0$ for $i=0,...,k(z)-1$ and thus

$$(a_i-a_{i+1})\eta_1(\OO_i)\leq (a_i-a_{i+1})\eta_2(\OO_i)M_{\eta_1/\eta_2}(\xi).$$

We recapitulate. We have a constant $C_1$ independent of $z,\xi,\eta_i$ such that

\begin{eqnarray*}
h_1(z)&\leq& C_1M_{\eta_1/\eta_2}(\xi)\left[\sum_{i=0}^{k(z)-1}(a_i-a_{i+1})\eta_2(\OO_i)+a_{k(z)}\eta_2(\OO_{k(z)})\right]\\
      &\leq& C_1M_{\eta_1/\eta_2}(\xi)\left[\sum_{i=1}^{k(z)}a_i(\eta_2(\OO_i)-\eta_2(\OO_{i-1}))+a_0\eta_2(\OO_0)\right]\\
			&\leq& C_1^2 M_{\eta_1/\eta_2}(\xi) h_2(z),
\end{eqnarray*}     
the third inequality coming from the lower bound given by Lemma \ref{decintrepr}. This completes the proof of Proposition \ref{ingredientfatou}.\quad \hfill $\square$

\section{Properties and some examples of $F$-harmonic functions}
\label{propertiesexamples}

We conclude this paper by giving some basic properties of $F$-harmonic functions, and by giving some examples coming from projective foliated bundles.

\subsection{Uniqueness of the $F$-harmonic function on a compact manifold}

The goal here is to prove Theorem \ref{uniqueFharmonic}: there is a unique $F$-harmonic function on $M$. We gave a much simpler proof of this fact in \cite{Al3} which is based on the bijective correspondence between Gibbs states on $T^1M$ for the geodesic flow and $F$-harmonic measures on $M$ of the form $hd\Leb$, with $f$ $F$-harmonic. The uniqueness of the Gibbs state then implies the uniqueness of the $F$-harmonic density. Since our initial motivation was to use a Theorem {\`a} la Fatou in order to prove this uniqueness result, we decided to give the longer proof.

\paragraph{Proof of Theorem \ref{uniqueFharmonic}.}

It is enough to prove that, up to a multiplicative constant, $h^F_0$ is the unique $F$-harmonic function on $N$ which is invariant by the action of the group $\pi_1(M)$. First note that it is clearly invariant by the equivariance property of Ledrappier's measures (see Theorem \ref{ledrappiermeasures}).

Now let $h$ be a $\pi_1(M)$-invariant $F$-harmonic function on $N$. There exists a finite measure $\eta$ on $N(\infty)$ such that $h$ possesses an integral representation
$$h(z)=\int_{N(\infty)}k^F(o,z;\xi)d\eta(\xi),$$
for $z\in N$. Let us write the Lebesgue decomposition of $\eta$ with respect to $\nu_o^F$: $\eta=f\nu^F_o+\eta_s$ where $f$ is $\nu_o^F$-integrable and $\eta_s$ is singular with respect to $\nu_o^F$.

First, we notice that the function $h$ is continuous and invariant by $\pi_1(M)$, which is realized as a cocompact subgroup of isometries of $N$. As a consequence $h$ is bounded. By the second part of Theorem \ref{theoremedefatou} the measure $\eta$ has no singular part with respect to $\nu_o^F$: we can write $\eta=f\nu_o^F$.

Using again the compactness of $M$ we see that any geodesic is at bounded distance to some broken line which links some points of the orbit of $o$ under $\pi_1(M)$. As a consequence, every point $\xi\in N(\infty)$ may be obtained as the nontangential limit of some sequence $(\gamma_i o)_{i\in\N}$, $\gamma_i\in\pi_1(M)$. Let $\xi\in N(\infty)$ and $(\gamma_i o)_{i\in\N}$ such an orbit.

By $\pi_1(M)$-invariance of $h$ and $h^F_0$ we get $h(o)/h^F_0(o)=\lim_{i\to\infty} h(\gamma_i o)/h^F_0(\gamma_i o)$. Thus the first part of Theorem \ref{theoremedefatou} implies that for $\nu_o^F$-almost every $\xi\in N(\infty)$, $h(o)/\widetilde{h}^F_0(o)=f(\xi)$.

Finally we find that $\eta=h(o)\nu_o^F$. Recall that $k^F(o,z;\xi)=d\nu^F_z/d\nu_o^F(\xi)$. Thus we get for all $z\in N$, $h(z)=[h(o)/h^F_0(o)]\mass(\nu_z^F)=[h(o)/h^F_0(o)]h^F_0(z)$, thus concluding the proof of the theorem.\quad \hfill $\square$

\subsection{Uniqueness of the integral representation}

We now turn to the proof of Theorem \ref{uniquedecompositionFharmonic} which states the uniqueness of the integral representation of $F$-harmonic functions. Let $\eta_1$, $\eta_2$ be two finite Radon measures on $N(\infty)$ and consider the two corresponding $F$-harmonic functions $h_1$ and $h_2$. Assume that $h_1(z)=h_2(z)$ for every $z\in N$. Write the Lebesgue decomposition $\eta_1=f\eta_2+\eta_s$ where $f$ is $\eta_2$-integrable and $\eta_s$ is singular with respect to $\eta_2$.

Then we have for every $z\in N$ $h_1(z)/h_2(z)=1$. In particular this quotient is everywhere bounded and by the last part of Theorem \ref{theoremedefatou} we get that there is no singular part: $\eta_1=f\eta_2$. By letting $z$ converge nontangentially to some $\xi\in N(\infty)$ we find that $f=1$ almost everywhere, completing the proof.\quad \hfill $\square$

\subsection{Examples of $F$-harmonic functions associated to projective foliated bundles}
\label{examples}

We now give some of the examples associated to projective foliated bundles which led us to consider this theory of $F$-harmonicity. In what follows, $M$ is still a closed and negatively curved manifold and $N$ its Riemannian universal cover. We consider a H{\"o}lder continuous potential $F:T^1M\to\R$.

\paragraph{Projective bundles.} Consider a representation $\rho:\pi_1(M)\to PSL_2(\C)$. We can \emph{suspend} the representation: $\pi_1(M)$ acts diagonally on $N\times\C\PP^1$ (on $N$ it acts by deck transformations and on $\C\PP^1$ by $\rho$). The quotient is a manifold endowed with a structure of projective $\C\PP^1$-bundle $\Pi:E\to M$ and with a foliation $\F$ transverse to the fibers whose holonomy representation is given by $\rho$ (see for example \cite{CL}). It is possible to \emph{parametrize} the leaves of $\F$ (which are covers of $M$) by the Riemannian structure of $M$ by lifting it by $\Pi$.

The ergodic study of the leaves of $\F$ was the main topic of \cite{Al3}. We defined $F$-\emph{harmonic measures} for $\F$ as measures on $E$ which have Lebesgue disintegration in the leaves, and whose local densities are $F$-harmonic. They generalize Garnett's harmonic measures for foliations \cite{Gar}.

Their existence was guaranteed by a bijective correspondence with a notion of Gibbs measures for the foliated geodesic flow. In the absence of transverse invariant measure by holonomy we proved that the $F$-harmonic measure is unique and is the solution of a natural problem of equidistribution of the leaves. The problem of equidistribution consists in looking at the accumulation points of weighted averages in large balls of leaves of $\F$ of the form
$$\mu^F_{x,R}=\proj_x\,_{\ast}\left(\frac{\kappa^F(o,y)\Leb_{|B(o,R)}}{\int_{B(o,R)}\kappa^F(o,y) d\Leb(y)}\right),$$
where $\kappa^F(o,y)=\int_o^yF$, $x\in E$ and $\proj_x:(N,o)\to (L_x,x)$ is the Riemannian universal cover of the leaf $L_x$. When $F=0$, we just look at the uniform weight on large balls and let the radius grow: it is a multidimensional analogue of Birkhoff averages. We obtained:

\begin{theorem}
\label{uniqueFharmonicmeasure}
Let $(\Pi,E,M,\C\PP^1,\F)$ be a projective foliated bundle over a closed Riemannian manifold $M$ with negative sectional curvature. Assume that the leaves are locally isometric to the base. Assume moreover that no probability measure on $\C\PP^1$ is invariant by the holonomy group. Then for any H{\"o}lder continuous potential $F:T^1M\to\R$, there is a unique $F$-harmonic measure for $\F$ denoted by $m_F$.

Moreover when the pressure of $F$ is positive, for any sequences $(R_n)_{n\in\N}$ of positive numbers tending to infinity and $(x_n)_{n\in\N}\in E^{\N}$ the measure $\mu^F_{x_n,R_n}$ converges to $m_F$ as $n$ tends to infinity.
\end{theorem}

\paragraph{Remark.} Let us mention that the hypothesis made on the pressure of $F$ is not restrictive since it is always possible to add a constant to $F$ and that $P(F+c)=P(F)+c$. Moreover when one adds a constant to $F$ one does not change the Gibbs kernel $k^F$, so the notion of $F$-harmonic measure remains unchanged.

\paragraph{Characteristic measure class and function of a leaf.} A $F$-harmonic measure provides a priori only \emph{local} $F$-harmonic densities in the plaques of $\F$. The holonomy inside the leaves could be an obstruction to extend these local densities. But we showed in \cite{Al1,Al2,Al3} that this obstruction is in fact void for a typical leaf. More precisely, let $(U_i)_{i\in I}$ be a finite covering of $M$ that trivialize the bundle, such that the intersection of any two of the $U_i$'s is connected. A $F$-harmonic measure $m_F$ reads in $\Pi^{-1}(U_i)\simeq U_i\times\C\PP^1$ as
$$(dm_F)_{|U_i}=h_i(z,x)d\Leb(z)d\nu_i(x),$$
where $\nu_i$ is a measure on $\C\PP^1$ and $h_i$ is a measurable function such that the maps $h_i(,.t)$ are $F$-harmonic functions of $U_i$. This formula shows that all the $\nu_i$ lie in the same class and it is possible to speak of a \emph{typical element of the fiber}. A lemma due to Ghys (\cite{Gh}, Lemme p.413) shows that if $p\in U_{i_0}$ and $x\in V_p=\Pi^{-1}(p)$ is $\nu_{i_0}$-typical then for any two paths $c_1,c_2$ inside $M$ that both start a $x$ and end at the same point inside, say $U_i$
$$\frac{d[\tau_{c_1}^{-1}\,_{\ast}\nu_i]}{d\nu_{i_0}}(x)=\frac{d[\tau_{c_2}^{-1}\,_{\ast}\nu_i]}{d\nu_{i_0}}(x),$$
where $\tau_c$ denotes the \emph{holonomy map} over the path $c$.

Hence for a typical $x$, the following function is well defined on $L_x$
$$H_x(z)=\frac{d[\tau_{c}^{-1}\,_{\ast}\nu_i]}{d\nu_{i_0}}(x) h_i(z),$$
where $z\in U_i$ and $c$ is \emph{any} path linking $x$ and $z$. 

The function $H_x$ is called the \emph{characteristic function} of the typical leaf $L_x$. It well defined up to a multiplicative constant. It has a unique integral representation associated to a finite Radon measure $\eta_x$ on $N(\infty)$. Here again the measure is defined up to a multiplicative constant, but the class is canonical. We call it the \emph{characteristic  measure class} of the typical leaf. This terminology is due to Matsumoto \cite{M}. We give in our context a generalization of a theorem of Matsumoto \cite{M} which is very much in the spirit of Theorem E of \cite{Al2}.

\begin{theorem}
\label{characteristic}
Let $(\Pi,E,M,\C\PP^1,\F)$ be a projective foliated bundle over a closed Riemannian manifold $M$ with negative sectional curvature. Assume that the leaves are locally isometric to the base. Assume moreover that no probability measure on $\C\PP^1$ is invariant by the holonomy group. Let $F:T^1M\to\R$ be any H{\"o}lder continuous potential and $m_F$ be the unique $F$-harmonic measure. Then:
\begin{itemize}
\item the characteristic class of a $m_F$-typical leaf is singular with respect to the class of Ledrappier's measures;
\item the characteristic function of a $m_F$-typical leaf is unbounded.
\end{itemize}
\end{theorem}

\begin{proof}
It is enough to copy the proof of \cite[Theorem E]{Al2}. The role of the main technical ingredient in the proof of this Theorem, which is \cite[Proposition 6.3]{Al2}, is played by \cite[Proposition 4.11]{Al3}. It will also be possible to prove that the characteristic class is singular with respect to the class of Ledrappier's measures if and only if the disintegration of the unique $F$-Gibbs measure for the foliated geodesic flow in the stable manifolds is singular with respect to the Gibbs class (see \cite{Al3} for the terminology), and that the latter is satisfied when there is no invariant measure.

Now that we know that the characteristic measure class is singular with respect to the Ledrappier's class we use the last part of our theorem {\`a} la Fatou in order to deduce that the characteristic function is unbounded.
\end{proof}

\paragraph{Family of $F$-harmonic measures.} We now turn to the disintegration of the unique $F$-harmonic measure in the fibers of $\Pi:E\to M$. Call $(m_{F,p})_{p\in M}$ the family of conditional measures. These conditional measures are the solution of a natural equidistribution problem of orbits of $\rho$ that we shall describe below.

Define the following distance function on $\pi_1(M)$, which acts on $N$ by isometries: $d(\gamma_1,\gamma_2)=\dist(\gamma_1 o,\gamma_2 o)$, where $o$ is some base point. Consider the following weight $\kappa^F(\gamma)=\kappa^F(o,\gamma o)$. Let $B_R$ denote the ball of center $Id$ and radius $R$ inside $\pi_1(M)$ for this distance $d$. Let $p$ be the projection of $o$ on $M$. Consider the family of measures given by the weighted counting measures
$$\theta_{F,R}=\frac{1}{\sum_{\gamma\in B_R}\kappa^F(\gamma)}\sum_{\gamma\in B_R}\kappa^F(\gamma)\delta_{\rho(\gamma)^{-1}x},$$
where $x\in\C\PP^1$. It is proven in \cite{Al3}

\begin{theorem}
\label{countingmeasures}
Let $M$ be a closed Riemannian manifold with negative sectional curvature. Consider a projective representation $\rho:\pi_1(M)\to PSL_2(\C)$ which leaves no probability measure invariant on $\C\PP^1$ and consider a H{\"o}lder continuous potential $F:T^1M\to\R$. Assume moreover that the potential $F$ has positive pressure. Then the measure $\theta_{F,R}$ converges to $m_{F,p}$ as $R$ tends to infinity.
\end{theorem}

It is possible \cite{Al3} to write the lift of $m_F$ to $N\times\C\PP^1$ as
$$\widetilde{m}_F=\widetilde{H}_x(z)\Leb_{N\times\{x\}}(z)\,\nu(x),$$
where $\nu$ is a finite measure on $\C\PP^1$ and $\widetilde{H}_x:N\to\R$ is a family of $F$-harmonic functions which varies measurably with $x$ which are well defined when $x\in A$, where $A\dans\C\PP^1$ is a set invariant by all holonomy transformations.

The functions $\widetilde{H}_x$ are exactly the lifts to the universal cover of the characteristic functions of typical leaves. These are unbounded $F$-harmonic functions by Theorem \ref{characteristic}.

Consider the conditional measures of $\widetilde{m}_F$ with respect to Lebesgue in the $\{z\}\times\C\PP^1$, $z\in N$. They read as $\widetilde{m}_z(x)=\widetilde{H}_x(z)\nu(x)$. Then, by analogy the definition of family of harmonic measures associated to a nonelementary representation of surface group given by Deroin and Dujardin \cite[Proposition 1.2]{DD} we can define the family of $F$-harmonic measures associated to any projective representation of a closed and negatively curved manifold.

\begin{theorem}
\label{Fharrepr}
Let $M$ be a closed Riemannian manifold with negative sectional curvature. Consider a projective representation $\rho:\pi_1(M)\to PSL_2(\C)$ which leaves no probability measure invariant on $\C\PP^1$ and consider a H{\"o}lder continuous potential $F:T^1M\to\R$. Then, up to a multiplicative constant, there exists a unique family of finite measures in $\C\PP^1$ $(\widetilde{m}_z)_{z\in N}$ such that the following holds:
\begin{enumerate}
\item it is $\rho$-equivariant: for every $\gamma\in\pi_1(M)$ and $z\in N$, $\rho(\gamma)_{\ast}\widetilde{m}_z=\widetilde{m}_{\gamma z}$;
\item it is $F$-harmonic: for every Borel set $B\dans\C\PP^1$ the function $z\mapsto\widetilde{m}_z(B)$ is $F$-harmonic.
\end{enumerate}

Moreover the quotient family is $(m_{F,p}/h_0^F(p))_{p\in M}$ where $(m_{F,p})_p$ is the family of conditional measures of the unique $F$-harmonic measure $m_F$ and $h_0^F$ denotes the unique $F$-harmonic function of the suspended foliation $\F$: it is the solution of the problem of equidistribution given by Theorem \ref{countingmeasures}.
\end{theorem}

These examples provide other natural examples of bounded $F$-harmonic functions on $N$ than just the trivial functions given by masses of Ledrappier measures.

\section*{Appendix. Borel density in the boundary of a Gromov hyperbolic space}

We now prove Borel's differentiation theorem (i.e. Theorem \ref{boreldifferentiation}) in a slightly more general context. This argument was developped by Tanaka in \cite{T} and is based on Bonk-Schramm's work on bilipschitz embeddings of hyperbolic spaces \cite{BS}. For more details about Gromov hyperbolicity, we refer to \cite{BHM,BS} and to the references therein.

\paragraph{Gromov hyperbolic spaces.} Say a geodesic metric space $(X,d)$ is \emph{$\delta$-hyperbolic} if for every $x,y,z,w\in X$ we have
$$(y|z)_w\geq\Min\{(x|z)_w,(z|y)_w\}-\delta,$$
where the Gromov product is defined as
$$(x|y)_z=\frac{d(z,x)+d(z,y)-d(x,y)}{2}.$$

\paragraph{Boundary.} Fix $o\in X$ and say a sequence $(x_n)_{n\in\N}$ is \emph{divergent} if $(x_n|x_m)_o\to\infty$ as $m,n\to\infty$. Say to sequences $(x_n)_{n\in\N}$ and $(y_n)_{n\in\N}$ are equivalent if $(x_n|y_n)_o\to\infty$ as $n\to\infty$. The set $X(\infty)$ of equivalence classes of divergent sequences shall be called the \emph{boundary of $X$}. This is a compact space. Gromov product can be extended to $X(\infty)$ via the formula
$$(\xi|\eta)_o=\Sup\left\{\limi_{n\to\infty}(x_n|y_n)_o\right\},$$
where the supremum is taken on all sequences $(x_n)_{n\in\N},(y_n)_{n\in\N}$ chosen in the classes of $\xi$ and $\eta$ respectively.

Obviously we can also define
$$(\xi|x)_o=\Sup\left\{\limi_{n\to\infty}(x_n|x)_o\right\},$$
where the supremum is taken on all sequences $(x_n)_{n\in\N}$ chosen in the class of $\xi$.

When $\alpha>0$ is small  enough there is a natural \emph{visual distance} on $X(\infty)$ denoted by $\rho_{\alpha}$ with the property that
$$C_{\alpha}^{-1}e^{\alpha(\xi|\eta)_o}\leq\rho_{\alpha}(\xi,\eta)\leq C_{\alpha}e^{\alpha(\xi|\eta)_o},$$
where $\xi,\eta\in X(\infty)$ and $C_{\alpha}$ is a constant greater than $1$ depending only on $\alpha$. The ball in $X(\infty)$ of radius $R>0$  and center $\xi$ for $\rho_{\alpha}$ shall be denoted by $B_{\alpha}(\xi,R)$.

\paragraph{Shadows.} Let $R>0$ and $x\in X$. The \emph{shadow} $\OO_R(o,x)$ is the set of $\xi\in X(\infty)$ such that $(\xi|x)_o\geq d(o,x)-R$. Next proposition may be found in \cite[Proposition 2.1]{BHM}.
\begin{proposition}
\label{shadowball}
Let $(X,d)$ be a $\delta$-hyperbolic space. For every $\tau>0$ there exist positive constants $C,R_0$ such that for every $R>R_0$, $x\in X$ and $\xi\in X(\infty)$ with $(\xi|o)_o\leq\tau$ we have,
$$B_{\alpha}\left(\xi, C^{-1} e^{\alpha(-d(o,x)+R)}\right)\dans \OO_R(o,x)\dans B_{\alpha}\left(\xi, C e^{\alpha(-d(o,x)+R)}\right).$$
\end{proposition}

The universal cover of a compact manifold with negative curvature is a Gromov hyperbolic space and shadows defined above are shadows as defined in \S \ref{defshadows} with a slightly different $R$.

\paragraph{Bilipschitz embeddings.} In \cite{As} (see also \cite[Section 9]{BS}), Assouad defined a notion of metric dimension (the so called \emph{Assouad dimension}), and proved that if a metric space $(Y,\rho)$ with finite Assouad dimension, then there exists a bilipschitz embedding of $(Y,\rho^p)$ in some Euclidean space $\R^n$ for some $p\in (0,1)$: this is \cite[Proposition 2.6]{As}. 

Bonk and Schramm  proved in \cite[Theorem 9.2]{BS} that $(X(\infty),\rho_{\alpha})$ has finite Assouad dimension when the $\delta$-hyperbolic space $(X,d)$ has \emph{bounded growth at some scale}. We refer to \cite{BS} for the formal definition: this condition is satisfied by every connected and simply connected Riemannian manifold \emph{whose sectional curvature is pinched between two negative constants}.

\begin{theorem}[Bilipschitz embedding]
\label{bonkschramm}
Let $(X,d)$ be a $\delta$-hyperbolic space with bounded growth at some scale. Then there exists $\alpha\in (0,1)$ such that there exists and a bilipschitz embedding 
$$\Phi:(X(\infty),\rho_{\alpha})\to(\R^n,\dist_{eucl})$$
for $n\in\N^{\ast}$, $\dist_{eucl}$ denoting the Euclidean distance.
\end{theorem}

\paragraph{Borel's differentiation theorem.} We can now state the main result.

\begin{theorem}
\label{BorelGromov}
Let $(X,d)$ be a $\delta$-hyperbolic space with bounded growth at some scale. Let $\alpha$ be given by Theorem \ref{bonkschramm}. 
Let $\eta_1,\eta_2$ be two finite Radon measures on $N(\infty)$. Write the Lebesgue decomposition of $\eta_1$ with respect to $\eta_2$ as:
$$\eta_1=f\eta_2+\eta_s,$$
where $f$ is $\eta_2$-integrable on $N(\infty)$, and $\eta_s$ is singular with respect to $\eta_2$.

Then, for $\eta_2$-almost every $\xi\in X(\infty)$, we have:
$$\lim_{r\to 0}\frac{\eta_1(B_{\alpha}(\xi,r))}{\eta_2(B_{\alpha}(\xi,r))}=f(\xi).$$

In particular, when $\eta_1$ and $\eta_2$ are singular, we have for $\eta_1$-almost every $\xi\in N(\infty)$:
$$\lim_{r\to 0}\frac{\eta_1(B_{\alpha}(\xi,r))}{\eta_2(B_{\alpha}(\xi,r))}=\infty.$$
\end{theorem}

\begin{proof}
The measures $\Phi_{\ast}\eta_1$ and $\Phi_{\ast}\eta_2$ are Radon measures of $\R^n$ and with Lebesgue decomposition
$$\Phi_{\ast}\eta_1=(f\circ\Phi^{-1})\,\,\Phi_{\ast}\eta_2+\Phi_{\ast}\eta_s.$$

Since the image by $\Phi$ of a $\rho_{\alpha}$-ball lies between two Euclidean balls whose radii have uniformly bounded ratio, the usual Borel's differentiation theorem in $\R^n$ (see \cite{Matt}) allows one to conclude the proof of the theorem.
\end{proof}

In order to prove Theorem \ref{boreldifferentiation}, note that by Proposition \ref{shadowball} any shadow lies between two $\rho_{\alpha}$-balls whose radii have uniformly bounded ratio and use the theorem above.

In order to prove Lemma \ref{maxfaiblementl1}, the fact that the Hardy-Littlewood function is weakly integrable, we use the exact same argument as well as the fact that such a result holds in Euclidean space for Radon measures: see \cite[Theorem 2.19]{Matt}.

\vspace{10pt}
\paragraph{Acknowledgments} This article is an improvement of the sixth chapter of my PhD thesis \cite{Al4}. I am happy to thank Christian Bonatti, Patrick Gabriel, Fran{\c c}ois Ledrappier and Barbara Schapira for many useful conversations. I also wish to thank the anonymous referee, who had found a gap in my previous attempt to prove Borel's differentiation theorem. Fran{\c c}ois Ledrappier first suggested me to use Markov partitions to prove Besikovich's property in the sphere at infinity, and then showed me Tanaka's manuscript: I am very indebted to him. Finally it is a pleasure to thank Ryokichi Tanaka for letting me use his elegant argument.

\begin{flushleft}
{\scshape S\'ebastien Alvarez}\\
Instituto Nacional de Matem\'atica Pura e Aplicada (IMPA)\\
Estrada Dona Castorina 110, Rio de Janeiro, 22460-320, Brasil\\
email: salvarez@impa.br
\end{flushleft}

\end{document}